\documentclass[a4,12pt]{amsart}
\usepackage{amsfonts}
\usepackage{amssymb}
\usepackage{amsmath}
\usepackage{amsthm}
\usepackage{fullpage}
\theoremstyle{plain}
\newtheorem{thm}{Theorem}[section]
\newtheorem{lem}[thm]{Lemma}
\theoremstyle{definition}
\newtheorem{defn}[thm]{Definition}
\newtheorem{ex}[thm]{Example}
\newtheorem{rem}[thm]{Remark}

\numberwithin{equation}{section}

\DeclareMathOperator{\Id}{Id}
\begin{document}
\title[eigenvalue]{ Existence of solutions of semilinear systems with gradient dependence via eigenvalue criteria}\thanks{Partially supported by G.N.A.M.P.A. - INdAM (Italy)}
\author{Filomena Cianciaruso}%
\address{Filomena Cianciaruso, Dipartimento di Matematica e Informatica, Universit\`{a}
della Calabria, 87036 Arcavacata di Rende, Cosenza, Italy}
\email{filomena.cianciaruso@unical.it}
\begin{abstract}
In this paper new criteria are established  for the existence of positive radial solutions of a semilinear elliptic system depending on the gradient. These criteria are determined by some relationships between the upper and lower bounds on suitable stripes of $\mathbb R^n$ of the nonlinearities of the system and the principal characteristic values of some associated linear Hammerstein integral operators. Moreover, using smoothing tools, the totality of the involved cone is established.
\end{abstract}
\maketitle

\section{Introduction}
In this paper we establish new criteria for the existence of positive radial solutions  for the system of BVPs
\begin{equation}\label{PDE}
\begin{cases}\
&-\Delta u =f_1(|x|,u,v,|\nabla u|,|\nabla v|)\text{ in } \Omega, \\
&-\Delta v=f_2(|x|,u,v,|\nabla u|,|\nabla v|)\text{ in } \Omega,\\
&u=0 \text{ on }\partial \Omega,\\
&v=0 \text{ on }|x|=R_0 \text{ and }\displaystyle\frac{\partial
v}{\partial r}=0 \text{ on }|x|=R_1,
\end{cases}
\end{equation}
where $\Omega=\{ x\in\mathbb{R}^n : R_0<|x|<R_1\}$ is an annulus,
$0<R_0<R_1<+\infty$, the nonlinearities $f_i$ are non-negative
continuous functions and $\dfrac{\partial}{\partial r}$ denotes (as
in ~~\cite{nirenberg}) differentiation in the radial direction
$r=|x|$.\\\
The problem of the existence of positive \emph{radial}
solutions of elliptic equations having nonlinearities that depend on  the gradient, subject to Dirichlet or mixed
 boundary conditions, has been investigated,
via different methods, by a number of authors, for example in  ~\cite{ave-mo-tor, bue-er-zu-fe, cia-pie, defig-sa-ubi, defig-ubi, fa-mi-pe, singh}. 
We seek radial solutions of the system \eqref{PDE} by means of an auxiliary system of nonlinear Hammerstein integral equations using the fixed point index theory and the invariance properties of the involved cone.\\
The existence of solutions for nonlinear Hammerstein integral equations or systems with nonlinearities with dependence on the first derivative has been studied in ~\cite{aga-ore-yan, ave-graef-liu, cia-pie, defig-sa-ubi, gra-kon-min, guo-ge, inf-min, janko, min-desou, min-sou, yang-kong, zima}. \\
In the recent paper ~\cite{cia-pie} in collaboration with Pietramala, we worked in  the Banach space $C_{\omega_1}^1[0,1]\times C_{\omega_2}^1[0,1]$, where the \emph{weights} $\omega_i$ are suitable nonnegative and continuous functions on $[0,1]$. 
In the special case $\omega(t)=t(1-t)$, the space $C_{\omega}^1[0,1]$ is utilized by Agarwal and others in ~\cite{aga-ore-yan}. \\
We defined a cone similar to the one defined in ~\cite{aga-ore-yan}  and, applying the index fixed point theory, we gave some conditions that assure the existence of positive solutions of the system \eqref{PDE}. These conditions relate the upper and lower bounds of the nonlinearities $f_i$ on suitable stripes and some computable constants  depending on the kernels of the associate Hammerstein integral operator and on the intervals in which the kernels are strictly positive.\\
\\
In this paper, via spectral theory,  we establish new existence results involving the principal characteristic values  of  auxiliary linear Hammerstein integral operators.\\
In this direction some results were obtained by Erbe ~\cite{erbe} and  Liu and Li ~\cite{ liu-li} in the case in which the kernel is symmetric. In 2006, Webb and Lan ~\cite{webb-lan} generalized these results using the permanence property of the fixed point index and requiring the uniqueness of the positive eigenvalues.  In 2009, Lan ~\cite{lan} obtained results for semipositone Hammerstein integral equations where the permanence property and the uniqueness of the positive eigenvalues are not used, but the results on the index being $1$ are obtained only for some open subsets. The first principal eigenvalue was also used by Li ~\cite{li} in the space $L^2$ requiring that the linear operator is normal and by Zhang and Sun ~\cite{zha-sun} for the $m-$point boundary value problems.\\
In 2011,  Lan and Lin studied, via spectral theory,  the existence of positive solutions of systems of Hammerstein integral equations in ~\cite{lan-lin, lan-lin-na} and Lan ~\cite{lan1} proved a new result for the existence of positive solutions of systems of second order elliptic boundary value problems. In 2015, Infante and Pietramala in ~\cite{inf-pie} established some criteria for the existence of solutions of systems of Hammerstein integral equations that involve a comparison with the spectral radii of some associated linear operators. \\ 
The Krein-Rutman Theorem, that is a celebrated result of  the spectral theory, requires the totality of the cone. In Section 2, using smoothing tools as the convolution operator and a sequence of mollifiers, we prove that the involve cone  is total.\\
Finally, an example shows that the results here obtained are applicable when the results proved in ~\cite{cia-pie} fail.
\section{The totality of the cone}
Let $\omega$ be a nonnegative continuous function from $[0,1]$ in $[0,1]$ and let $C_{\omega}^1[0,1]$ be the functions space defined by
$$
C_{\omega}^1[0,1]=\{w \in C[0,1]: w \text{ is continuous differentiable on }]0,1[ \text{ with }\sup_{t \in ]0,1[}\omega(t)|w'(t)|<+\infty\}.
$$
Set
$$|| w|| _{\infty}:=\underset{t\in [ 0,\,1]\,}{\max }|w(t)|,\,\,\,\|w'\|_{\omega}=\displaystyle\sup_{t\in ]0,1[}\omega (t)|w'(t)|,$$
it can  be verified that $C_{\omega}^1[0,1]$, equipped with the norm
$$|| w|| =\max \left\{ || w|| _{\infty},\|w'\|_{\omega}\right\} ,$$  is a Banach space (the proof follows as in ~\cite{aga-ore-yan}).\\
Fixed $[a,b] \subseteq [0,1]$, $0<c<1$,
 let $\mathcal K_{\omega}$ be the cone in $ C_{\omega}^{1}[0,1]$ defined, in a similar way as in  ~\cite{aga-ore-yan}, by
\begin{eqnarray*}\label{cone}
\mathcal K_{\omega}:=\left\{ w\in C_{\omega}^{1}[0,\,1]:w\geq 0,\,\underset{t\in [ a,b]}%
{\min }w(t)\geq c|| w|| _{\infty},\,\,  || w|| _{\infty}\geq \|w'\|_{\omega}\right\}.
\end{eqnarray*}
Note that the functions in $\mathcal K_{\omega}$ are strictly positive on the
sub-interval $[a,b]$ and that, for $w \in \mathcal K_{\omega}$, the equality $\|w\|=\|w\|_{\infty}$ holds.\\
\\
We prove here that the cone $\mathcal K_{\omega}$ is total, i.e. $$C^1_{\omega}[0,1]=\overline{\mathcal K_{\omega}-\mathcal K_{\omega}}\,.$$  To the best of our knowledge, this property has not been investigated.\\
Let
\begin{eqnarray*}
\mathcal P_{\omega}:=\{w\in C_{\omega}^1[0,1]:\ w(t)\geq 0 \text{ for }t\in [0,1]\}
\end{eqnarray*}
be the positive cone in $C_{\omega}^1[0,1]$. Firstly we prove that
\begin{lem} $\mathcal P_{\omega}=\mathcal K_{\omega}-\mathcal K_{\omega}$.
\end{lem}
\begin{proof} An element  $w \in \mathcal P_{\omega}/\{0\}$  can be rewritten, for $t\in[0,1]$, as:
$$
w(t)=\beta w(t) +\gamma \|w\|_{\infty}-((\beta-1)w(t)+\gamma \|w\|_{\infty}),
$$
where the constants $\beta$ and $\gamma$, depending on $w$, are defined by
\begin{equation*}
\beta=\begin{cases}1,\,\,\,\,\,\,\,\,\,\,\,\,\,\,\text{ if }\|w'\|_{\omega}\le \|w\|_{\infty}\\
\displaystyle\frac{\|w'\|_{\omega }}{\|w\|_{\infty}},\text{ if }\|w'\|_{\omega }> \|w\|_{\infty}
\end{cases}
\end{equation*}
and 
\begin{equation*}
\gamma=\max\left\{\frac{\|w'\|_{\omega }}{\|w\|_{\infty}}-1, \frac{c}{1-c}\right\}\beta.
\end{equation*}
Set, for $t\in [0,1]$, $\varphi(t):=\beta w(t) +\gamma \|w\|_{\infty}$ and $\psi(t):=(\beta-1)w(t)+\gamma \|w\|_{\infty}$, then $w=\varphi-\psi$; now we prove that $\varphi, \psi \in \mathcal K_{\omega}$.\\
It is clear that $\|\varphi\|_{\infty}=(\beta+\gamma)\|w\|_{\infty}$ and $\|\psi\|_{\infty}=(\beta-1+\gamma)\|w\|_{\infty}$. Since the function $\displaystyle{\frac{\gamma}{\beta + \gamma}}$ is nondecreasing with respect to $\gamma$, we have, for $t\in [0,1]$, 
$$
\varphi(t)\geq \gamma \|w\|_{\infty}=\frac{\gamma}{\beta +\gamma}\|\varphi\|_{\infty}\geq \frac{ \frac{c\beta}{1-c}}{\beta + \frac{c\beta}{1-c}}\|\varphi\|_{\infty}=c\|\varphi\|_{\infty}
$$
and
$$
\psi(t)\geq \gamma \|w\|_{\infty}=\frac{\gamma}{\beta-1 +\gamma}\|\psi\|_{\infty}\geq \frac{ \frac{c\beta}{1-c}}{\beta -1+ \frac{c\beta}{1-c}}\|\psi\|_{\infty}=\frac{c\beta}{\beta-(1-c)}\|\psi\|_{\infty}\geq c\|\psi\|_{\infty}.
$$
Now we prove the conditions on the derivatives.\\
In the case  $\|w'\|_{\omega }\le \|w\|_{\infty}$, one has $\beta=1$, $\gamma=\displaystyle\frac{c}{1-c}.$  Then  we have
$$
\|\varphi'\|_{\omega }=\|w'\|_{\omega }\le \|w\|_{\infty}=\frac{1}{1+\frac{c}{1-c}}\,\|\varphi\|_{\infty}=(1-c)\|\varphi\|_{\infty}<\|\varphi\|_{\infty}
$$
and 
$$
\|\psi'\|_{\omega}=0<\|\psi\|_{\infty}.
$$
In the case  $\|w'\|_{\omega }> \|w\|_{\infty}$,  one has $\beta=\displaystyle\frac{\|w'\|_{\omega }}{\|w\|_{\infty}}$ and $\beta^2\leq \beta+\gamma$; then we have 
$$
\|\varphi'\|_{\omega }=\beta\|w'\|_{\omega}=\frac{\|w'\|_{\omega }^2}{\|w\|_{\infty}}=\beta^2\|w\|_{\infty}\le (\beta+\gamma)\|w\|_{\infty}=\|\varphi\|_{\infty}
$$
and 
$$
\|\psi'\|_{\omega }=(\beta-1)\|w'\|_{\omega }=\beta^2\|w\|_{\infty}-\|w'\|_{\omega }\leq(\beta+\gamma)\|w\|_{\infty}-\|w\|_{\infty}  =(\beta+\gamma-1)\|w\|_{\infty}=\|\psi\|_{\infty}.
$$

\end{proof}
To prove that $\mathcal K_{\omega}$ is total, for  each fixed $w \in C_{\omega}^1[0,1]$  it is need to construct two sequences $(\varphi_n)_{n \in \mathbb N}, (\psi_n)_{n \in \mathbb N} \in \mathcal K_{\omega}$ such that $\varphi_n -\psi_n$ converges to $w$ in $ C_{\omega}^{1}[0,\,1]$. To do this, two tools are used: a sequence of mollifiers and the convolution operation $*$.\\
The convolution can be viewed as a smoothing operation: in fact, the convolution of two functions is differentiable as many times as the two functions are. \\
A sequence $(\rho_n)_{n \in \mathbb N}$ of nonnegative functions in the space $C_c^{\infty}(\mathbb R)$ of the functions with compact support is said a \textit{sequence of mollifiers} if the support of $\rho_n$ is contained in $\left[-\frac{1}{n},\frac{1}{n}\right]$ and $\displaystyle \int_{\mathbb R}\rho_n=1$. Since our functions $w$ are defined in $[0,1]$, we construct a sequence of mollifiers starting by a nonnegative function $\rho \in C_c^{\infty}(\mathbb R)$ with support in $[0,1]$.\\
\begin{thm}\label{total}
 The cone $$\mathcal K_{\omega}:=\left\{ w\in C_{\omega}^{1}[0,\,1]:w\geq 0,\,\underset{t\in [ a,b]}%
{\min }w(t)\geq c || w|| _{\infty},\,\,  || w|| _{\infty}\geq \|w'\|_{\omega}\right\}
$$ is total.
\end{thm}
\begin{proof} Let  $\rho \in C_c^{\infty}(\mathbb R)$ be a nonnegative function with support in $[0,1]$ and therefore $\rho(0)=0$. Let $( \rho_n)_{n \in \mathbb N}\subseteq C_c^{\infty}(\mathbb R)$  be a sequence of mollifiers defined by\begin{equation}\label{rhon}\rho_n(x)=\frac{n}{\int_{\mathbb R}\rho(y)dy}\rho(nx)\,\,;\end{equation}
 note that the support of $\rho_n$ is contained in $\left[0,\frac{1}{n}\right]$. \\
Let $w \in C_{\omega}^1[0,1]\setminus\{0\}$ be fixed; we construct a sequence $(w_n)_{n \in \mathbb N} \subseteq C^1[0,1]$ converging to $w$ in $C^1_{\omega}[0,1]$.\\
We discuss two cases.\\
{\bf Case I.} $w(0)=0$.\\
For $n \in \mathbb N$,  we define, for $t \in [0,1]$,
\begin{equation}\label{wn}
w_n(t)=:(\rho_n * w)(t)=\int_0^{t} w(t-y)\rho_n(y)dy=\int_0^t \rho_n(t-y)w(y)dy=(w*\rho_n)(t).
\end{equation}
The function $w_n$ is continuous on $[0,1]$ and from \eqref{wn} it follows that 
$$
w_n'(t)=\int_0^t \rho_n'(t-y)w(y)dy+\rho_n(0)w(t)=(\rho'_n*w)(t)+\rho_n(0)w(t),
$$
i.e. $w_n \in C^1[0,1]$.\\
Moreover, integrating by parts, it follows that 
\begin{eqnarray*}
\displaystyle w'_n(t)&=&\rho_n(t)w(0)-\rho_n(0)w(t)+\int_0^t \rho_n(t-y)w'(y)dy+\rho_n(0)w(t)\\
\displaystyle&=&\int_0^t \rho_n(t-y)w'(y)dy=\int_0^t w'(t-y)\rho_n(y)dy.
\end{eqnarray*}
From the last equality,  it follows that the  improper integrals are finite.\\
Now the proof is divided in more steps.\\
{\bf Step 1.} $\displaystyle \lim_{n \to +\infty}\|w_n-w\|_{\infty}=0$.\\
Since $\displaystyle \int_0^{\frac{1}{n}}\rho_n(x)=1$ and the support of $\rho_n$ is in $\displaystyle\left[0, \frac{1}{n}\right]$, one has
\begin{equation*}
|w_n(t)-w(t)|=\left|\int_0^t w(t-y)\rho_n(y)dy -w(t)\right|=\left|\int_0^t w(t-y)\rho_n(y)dy -\int_0^{\frac{1}{n}}w(t)\rho_n(y)dy\right|
\end{equation*}
\begin{equation*}
=\begin{cases}\displaystyle\left|\int_0^t (w(t-y)-w(t))\rho_n(y)dy -\int_t^{\frac{1}{n}}w(t)\rho_n(y)dy\right|,\,\, \text{ if }0\leq t\leq \frac{1}{n}\cr
\\
\displaystyle\left|\int_0^{\frac{1}{n}}(w(t-y)-w(t))\rho_n(y)dy\right| ,\,\,\,\,\,\,\,\,\,\,\,\,\,\,\,\,\,\,\,\,\,\,\,\,\,\,\,\,\,\,\,\,\,\,\,\,\,\,\,\,\,\,\,\,\,\,\,\,\,\,\,\text{ if }\frac{1}{n}<t \leq 1\cr
\end{cases}
\end{equation*}
\begin{equation*}
\leq \begin{cases}\displaystyle\int_0^t |w(t-y)-w(t)|\rho_n(y)dy +\|w\|_{\infty}\int_t^{\frac{1}{n}}\rho_n(y)dy,\,\, \text{ if }0\leq t\leq \frac{1}{n}\cr
\\
\displaystyle\int_0^{\frac{1}{n}}|w(t-y)-w(t)|\rho_n(y)dy ,\,\,\,\,\,\,\,\,\,\,\,\,\,\,\,\,\,\,\,\,\,\,\,\,\,\,\,\,\,\,\,\,\,\,\,\,\,\,\,\,\,\,\,\,\,\,\,\,\,\,\,\,\, \text{ if }\frac{1}{n}<t\leq 1.\cr
\end{cases}
\end{equation*}
From the uniform continuity of $w$ in $[0,1]$,  it follows that for each $\epsilon>0$ fixed there exists $\eta>0$ such that
$$t_1,t_2 \in [0,1],\, |t_1-t_2|<\eta \text{ implies }|w(t_1)-w(t_2)|<\frac{\epsilon}{2};$$
consequently, for $t \in [0,1]$, $y \in \big[0,\frac{1}{n}\big]$  and $n \geq \eta^{-1}$, one has
 $$|w(t-y)-w(t)|<\frac{\epsilon}{2} .$$
Moreover, since $\rho(0)=0$, by \eqref{rhon} and the Mean Integral Theorem it follows that, for  $t\in\left [0,\frac{1}{n}\right]$,
$$\lim_{n \to +\infty}\int_t^{\frac{1}{n}}\rho_n(y)dy=0.$$
Then there exists $\overline{n} \in \mathbb N$ such that,  for $n\geq \overline{n}$,
$$\|w\|_{\infty}\int_t^{\frac{1}{n}}\rho_n(y)dy<\frac{\epsilon}{2}.
$$
Set $n_0:=\displaystyle\max\left\{ [\eta^{-1}]+1,\overline{n}\right\}$, we have, for $n \geq n_0$, 
$$
\|w_n-w\|_{\infty}\leq \frac{\epsilon}{2}\int_0^{\frac{1}{n}}\rho_n(y)dy+\frac{\epsilon}{2}=\epsilon. 
$$
\\
{\bf Step 2.} $\displaystyle \lim_{n \to +\infty}\|w'_n-w'\|_{\omega}=0$.\\

Since $$\lim_{n \to +\infty}\|w'_n-w'\|_{\omega}=\lim_{n \to +\infty}\lim_{\delta \to 0}\max_{t \in [\delta, 1-\delta]}\omega(t)|w'_n(t)-w'(t)|,$$ we evaluate, for $t\in [\delta,1-\delta]$,
\begin{eqnarray*}
&\displaystyle\omega(t)|w'_n(t)-w'(t)|=\left|\int_0^t w'(t-y)\rho_n(y)dy-\int_0^{\frac{1}{n}}w'(t)\rho_n(y)dy\right|\\
&\displaystyle=\left|\int_0^t \omega(t)w'(t-y)\rho_n(y)dy-\int_0^{\frac{1}{n}}\omega(t)w'(t)\rho_n(y)dy\right|
\end{eqnarray*}
\begin{equation*}
=\begin{cases}\displaystyle \left|\int_0^t \omega(t)(w'(t-y)-w'(t))\rho_n(y)dy-\int_t^{\frac{1}{n}}\omega(t)w'(t)\rho_n(y)dy\right|,\,\, \text{ if }\delta\leq t \leq \frac{1}{n}\cr\\
\displaystyle \left|\int_0^{\frac{1}{n}} \omega(t)(w'(t-y)-w'(t))\rho_n(y)dy \right|,\,\,\,\,\,\,\,\,\,\,\,\,\,\,\,\,\,\,\,\,\,\,\,\,\,\,\,\,\,\,\,\,\,\,\,\,\,\,\,\,\,\,\,\,\,\,\,\,\,\,\,\,\,\,\,\,\,\,\,\, \text{ if }\frac{1}{n}<t\leq 1-\delta\cr
\end{cases}
\end{equation*}
\begin{equation*}
\leq \begin{cases}\displaystyle \int_0^t \omega(t)|w'(t-y)-w'(t)|\rho_n(y)dy+\|w'\|_{\omega}\int_t^{\frac{1}{n}}\rho_n(y)dy,\,\, \,\,\,\,\,\,\,\,\,\,\text{ if }\delta\leq t \leq \frac{1}{n}\cr\\
\displaystyle \int_0^{\frac{1}{n}} \omega(t)|w'(t-y)-w'(t)|\rho_n(y)dy ,\,\,\,\,\,\,\,\,\,\,\,\,\,\,\,\,\,\,\,\,\,\,\,\,\,\,\,\,\,\,\,\,\,\,\,\,\,\,\,\,\,\,\,\,\,\,\,\,\,\,\,\,\,\,\,\,\,\,\,\,\,\,\,\text{ if }\frac{1}{n}<t\leq 1-\delta. \cr
\end{cases}
\end{equation*}
Let $\epsilon>0$ be fixed; since $w'$ is uniformly continuous on compact intervals, there exists $\eta_1>0$ such that 
$$
\text{if }t_1,t_2 \in [\delta,1-\delta],\, |t_2|<\eta_1 \text{ then }\omega(t_1)|w'(t_1-t_2)-w'(t_1)|<\frac{\epsilon}{2};
$$
moreover, there exists $\tilde{n} \in \mathbb N$ such that, for $n\geq \tilde{n}$,
$$\|w'\|_{\omega}\int_t^{\frac{1}{n}}\rho_n(y)dy<\frac{\epsilon}{2}.
$$
Then, set $n_1:=\displaystyle\max\left\{ [\eta_1^{-1}]+1,\tilde{n}\right\}$, for $n \geq n_1$,
as above, one obtains
\begin{eqnarray*}
&\displaystyle\max_{t \in [\delta,1-\delta]}\omega(t)|w'_n(t)-w'(t)|< \epsilon,
\end{eqnarray*}
i.e. for every $\delta \in ]0,1[$
\begin{equation}\label{dini}
\lim_{n \to +\infty}\max_{t \in [\delta,1-\delta]}\omega(t)|w'_n(t)-w'(t)|=0.
\end{equation}
Since the function $\displaystyle \max_{t \in [\delta,1-\delta]}\omega(t)|w'_n(t)-w'(t)|$ is nonincreasing with respect to $\delta$, by Dini Theorem  it follows that the limit in \eqref{dini} is uniform in the compact subintervals of $]0,1[$.  \\
Moreover, for every $n \in \mathbb N$ there exists $$\lim_{\delta \to 0}\max_{t \in [\delta,1-\delta]}\omega(t)|w'_n(t)-w'(t)|=\sup_{t \in ]0,1[}\omega(t)|w'_n(t)-w'(t)|;$$
then,  by the Inversion Limit Theorem, for  $t$ in a compact subinterval of $]0,1[$ one has
$$
\lim_{n \to +\infty}\lim_{\delta\to 0}\max_{t \in [\delta,1-\delta]}\omega(t)|w'_n(t)-w'(t)|=\lim_{\delta\to 0}\lim_{n \to +\infty}\max_{t \in [\delta,1-\delta]}\omega(t)|w'_n(t)-w'(t)| =0
$$
and consequently $\displaystyle \lim_{n \to +\infty}\|w'_n-w'\|_{\omega}=0$.  \\
\smallskip

\noindent
{\bf Step 3.} $w_n\in \mathcal K_{\omega}-\mathcal K_{\omega}$.\\
The function $w$ can be rewritten as  $w=w^+-w^-$, where $w^+$ and $w^-$ are respectively the positive and negative parts of $w$; then, for $n \in \mathbb N$
\begin{equation}\label{wn+}
w_n=w_n^+-w_n^-, 
\end{equation}
where
$$
w_n^+(t)=:(\rho_n * w^+)(t)=\int_0^tw^+(t-y)\rho_n(y)dy\text{ and }
w_n^-(t)=:(\rho_n * w^-)(t)=\int_0^tw^-(t-y)\rho_n(y)dy.
$$ 
Since the functions $w_n^+$ and $w_n^-$  belong to $C^1[0,1]$ and are positive, they belong to the cone $\mathcal P_{\omega}$ and therefore there exist $\varphi_n^+, \varphi_n ^-, \psi_n^+, \psi_n^-\in \mathcal K_{\omega}$ such that 
$$
w_n^+=\phi_n^+-\psi_n^+,\,\,\,\,\,\,w_n^-=\phi_n^--\psi_n^-.
$$
By the identity \eqref{wn+} the assert follows.\\
\smallskip

\noindent
{\bf Case II.} Case $w(0)\not=0$.\\
With a translation we return to the previous case. In fact, set $\tilde{w}=w-w(0)$, the sequence $\tilde{w}_n$ defined as in \eqref{wn} converges in $C^1_{\omega}[0,1]$ to  $\tilde{w}$ and the sequence $\overline{w}_n=\tilde{w}_n+w(0)\to w$ in $C^1_{\omega}[0,1]$.
Let $ \tilde{\phi}_n,\tilde{\psi}_n\in \mathcal K_{\omega}$ be such that $\tilde{w}_n=\tilde{\phi}_n-\tilde{\psi}_n$; then $\overline{w}_n=\phi_n-\psi_n$ where
$$
\phi_n=\begin{cases}\tilde{\phi}_n+w(0),  \,\,\text{ if }w(0)>0\cr
\tilde{\phi}_n,  \,\,\,\,\,\,\,\,\,\,\,\,\,\,\,\,\,\,\,\,\,\,\text{ if }w(0)<0
\end{cases}  
\text{ and }\,\,
\psi_n=
\begin{cases}\tilde{\psi}_n,   \,\,\,\,\,\,\,\,\,\,\,\,\,\,\,\,\,\,\,\,\,\text{ if }w(0)>0\cr
\tilde{\psi}_n-w(0),  \,\,\text{ if }w(0)<0.
\end{cases}  
$$
\end{proof}
\smallskip

\section{Auxiliar results}
In this section we recall notations and results  of ~\cite{cia-pie}  that will be useful in the sequel.\\
By a \emph{radial solution}  of the elliptic system
$$
\begin{cases}\
&-\Delta u =f_1(|x|,u,v,|\nabla u|,|\nabla v|)\text{ in } \Omega, \\
&-\Delta v=f_2(|x|,u,v,|\nabla u|,|\nabla v|)\text{ in } \Omega,\\
&u=0 \text{ on }\partial \Omega,\\
&v=0 \text{ on }|x|=R_0 \text{ and }\displaystyle\frac{\partial
v}{\partial r}=0 \text{ on }|x|=R_1,
\end{cases}
$$
we mean a solution of the associate system of ODEs
\begin{equation}\label{1syst}
\begin{cases}
&-u''(t) = g_1(t,u(t),v(t),|u'(t)|,|v'(t)|) \quad \text{ in } [0,1], \\
&-v''(t) =g_2(t,u(t),v(t),|u'(t)|,|v'(t)|) \quad \text{ in } [0,1],\\
& u(0)=u(1)=v(0)=v'(1)=0,
\end{cases}
\end{equation}
where, for $t\in [0,1]$, $g_i$ is the nonnegative function given by
\begin{equation*} 
g_i(t,u(t),v(t),|u'(t)|,|v'(t)|) :=p(t)
f_i\left(r(t),u(t),v(t),\frac{|u'(t)|}{|r'(t)|},\frac{|v'(t)|}{|r'(t)|}\right),
\end{equation*}
$p$ is defined by
\begin{equation*}
p(t):=\begin{cases}
r^2(t) \log^2(R_1/R_0), \ & n=2\\
\left(\frac{R_0R_1\left(R_1^{n-2}-R_0^{n-2}\right)}{n-2}\right)^2\,\frac{1}{\left(R_1^{n-2}-(R_1^{n-2}-R_0^{n-2})t\right)^{\frac{2(n-1)}{n-2}}},\
&n\geq 3
\end{cases}
\end{equation*} 
and  $r$ is defined by (see ~\cite{defig-ubi, dolo3})
\begin{equation*}
r(t):=\begin{cases}
R_1^{1-t}R_0^{t},\ &n=2\\
\left(\frac{A}{B-t}\right)^{\frac{1}{n-2}},\ &n\geq 3
\end{cases}
\end{equation*}
with
$$
A=\frac{(R_0R_1)^{n-2}}{R_1^{n-2}-R_0^{n-2}}\,\,\,\,\text{
 and  }\,\,\,\,B=\frac{R_1^{n-2}}{R_1^{n-2}-R_0^{n-2}}.
$$
\smallskip

Fix $$\omega_1(t)=t(1-t),\,\,\,\,\,\,\,\,\omega_2(t)=t,$$
consider the product space $C_{\omega_1}^1[0,1]\times C_{\omega_2}^1[0,1]$ equipped with the norm (with an abuse of notation)
$$
\|(u_1,u_2)\|=\max\{\|u_1\|,\|u_2\|\}.
$$
We search the solutions of the system \eqref{1syst} as fixed points of the compact operator $\mathcal T$ in $C_{\omega_1}^1[0,1]\times C_{\omega_2}^1[0,1]$ defined by 
\begin{equation}\label{operT}
\begin{array}{c}
\mathcal T(u,v)(t):=\left(
\begin{array}{c}
\mathcal T_{1}(u,v)(t) \\
\mathcal T_{2}(u,v)(t)%
\end{array}%
\right)  =\left(
\begin{array}{c}
\displaystyle\int_{0}^{1}k_1(t,s)g_{1}(s,u(s),v(s),|u'(s)|,|v'(s)|)\,ds \\
\displaystyle\int_{0}^{1}k_2(t,s)g_2(s,u(s),v(s),|u'(s)|,|v'(s)|)ds%
\end{array}%
\right)%
\end{array}\,,
\end{equation}%
where the Green's functions $k_i$  are given by
\begin{equation*} \label{ki}k_1(t,s)=\begin{cases} s(1-t),&0\leq s\leq t\leq 1\cr
t(1-s),&0\leq t\leq s\leq 1\cr\end{cases}
\,\,\text{ and }\,\,\,\,\,\,\,k_2(t,s)=\begin{cases} s,&0\leq s\leq t\leq 1\cr
t,&0 \leq t\leq s\leq 1.\cr\end{cases}
\end{equation*}
Now we resume some known properties of the functions $k_i$ that will be use in the sequel.
\begin{itemize}
 \item[(1)]   The kernel $k_1$ is positive and continuous in  $[0,1]\times  [0,1]$. Moreover, for $[a_1,b_1] \subset(0,1)$, take
$$\phi_1(s):=\sup_{t\in [0,1]} k_1(t,s)= k_1(s,s)=s(1-s) \text
{ and } c_1:=\min\{a_1, 1-b_1\}\,,
$$
\begin{equation}\label{k1}
k_1(t,s) \leq \phi_1(s)\ \text{ for } t,s\in [0,1] 
,\,\,k_1(t,s) \geq c_1\, \phi_1 (s) \text{ for }(t,s)\in
[a_1,b_1]\times [0,1].
\end{equation}
\item[(2)]  The function $k_1(\cdot,s)$ is derivable in $\tau \in [0,1]$, with
$$
\dfrac{\partial k_1}{\partial t}(t,s)=\begin{cases} -s,&0\leq s<
t\leq 1\cr 1-s, & 0\leq t< s\leq 1,\cr\end{cases}
$$
for  $\tau \in [ 0,1]$ 
\begin{equation*}
\lim_{t\rightarrow \tau }\left| \frac{\partial k_1}{\partial
t}(t,s)- \frac{\partial k_1}{\partial t}(\tau ,s)\right|
=0,\;\text{ for almost every}\,s\in [ 0,1]
\end{equation*}
and
$$
 \left|\frac{\partial k_1}{\partial t}(t,s)\right|\leq \psi_1(s):=\max\{s,1-s\}\text{    for   }t\in [0,1]\text{     and almost every }\,s\in [0,1].
 $$
 \item[(3)]   The kernel $k_2$ is positive and continuous in  $[0,1]\times  [0,1]$. Moreover, for $[a_2,b_2] \subset(0,1]$, take
$$\phi_2(s):=\sup_{t\in [0,1]} k_2(t,s)= k_2(s,s)=s \text{ and }c_2:=a_2\,,
$$
\begin{equation}\label{k2}
k_2(t,s) \leq \phi_2(s)\ \text{ for } t,s\in [0,1] 
,\,\,k_2(t,s) \geq c_2\, \phi_2 (s) \text{ for }(t,s)\in
[a_2,b_2]\times [0,1].
\end{equation}
 \item[(4)]  The function $k_2(\cdot,s)$ is derivable in $\tau \in [0,1]$, with
$$
\dfrac{\partial k_2}{\partial t}(t,s)=\begin{cases} 0,&0\leq s<
t\leq 1\cr 1, & 0\leq t< s\leq 1,\cr\end{cases}
$$
and for  $\tau \in [ 0,1]$ 
\begin{equation*}
\lim_{t\rightarrow \tau }\left| \frac{\partial k_2}{\partial
t}(t,s)- \frac{\partial k_2}{\partial t}(\tau ,s)\right|
=0,\;\text{ for almost every}\,s\in [ 0,1].
\end{equation*}
Moreover
$$
 \left|\frac{\partial k_2}{\partial t}(t,s)\right|\leq 1:=\psi_2(s)\text{    for   }t\in [0,1]\text{     and almost every }\,s\in [0,1].
 $$
\end{itemize}
By direct calculation we obtained
\begin{equation}\label{m}
m_1:=\left(\sup_{t\in [0,1]}\int_0^1k_1(t,s)\,ds\right)^{-1}=8,
\,\,\,\,\,\, m_2:=\left(\sup_{t\in
[0,1]}\int_0^1k_2(t,s)\,ds\right)^{-1}=2,
\end{equation}
\begin{equation}\label{M1}
M_1:=\left(\inf_{t\in[a_1,b_1]}\int_{a_1}^{b_1}k_1(t,s)ds\right)^{-1}=\begin{cases}\frac{2}{a_1(b_1-a_1)(2-a_1-b_1)},&\mbox
{ if }a_1+b_1\leq
1 \cr\\
\frac{2}{(1-b_1)(b_1^2-a_1^2)},& \mbox {if }a_1+b_1> 1,
\cr\end{cases} \,\,\,\,\,\,
\end{equation}
\begin{equation*}\label{M2}
M_2:=\left(\inf_{t\in[a_2,b_2]}\int_{a_2}^{b_2}k_2(t,s)ds\right)^{-1}=\frac{1}{a_2(b_2-a_2)}.
\end{equation*}
\smallskip

The following existence result for the system \eqref{PDE} is established in ~\cite{cia-pie}.
\begin{thm}~\cite{cia-pie}\label{ellyptic}
 Suppose that, for $i=1,2$, there exist
$\rho_i, s_i\in (0,+\infty )$, with $\rho _{i}<c_i\,s
_{i}$, such that the following conditions hold
\begin{equation*}\label{pde2}
\sup_{\Omega^{\rho_1,\rho_2}}
f_i(r,w_1,w_2,z_1,z_2)<\frac{m_i}{\displaystyle\sup_{t \in
[0,1]}p(t)}\,\rho_i
\end{equation*}
and
\begin{equation*}\label{pde3}
\inf_{A_i^{s_1,s_2}}
f_i(r,w_1,w_2,z_1,z_2)>\frac{M_i}{\displaystyle\inf_{t
\in [a_i,b_i]}p(t)}\,s_i,
\end{equation*}
where 
\begin{align*}
\Omega^{\rho_1,\rho_2}&=[ R_0,R_1]\times \left [0,
\rho_1\right]\times\left [0, \rho_2\right]\times
\left[0, +\infty\right)^2,\\
A_1^{s_1,s_2}&=[
\min\{r(a_1),r(b_1)\},\max\{r(a_1),r(b_1)\}]\times\left[s_1,\frac{s_1}{c_1}\right]\times\left[0,\frac{s_2}{c_2}\right]\times\left[0, +\infty\right)^2,\\
A_2^{s_1,s_2}&=[\min\{r(a_2),r(b_2)\},\max\{r(a_2),r(b_2)\}]\times\left[0,\frac{s_2}{c_2}\right]\times\left[s_2,\frac{s_2}{c_2}\right]\times\left[0, +\infty\right)^2.\\
\end{align*}
Then the system (\ref{PDE}) has at least one positive radial
solution. 
\end{thm}
\smallskip

The following theorem follows from classical
results about fixed point index (more details can be seen, for example, in
~\cite{Amann-rev, guolak}).
\begin{thm} \label{index}Let $K$ be a cone in an ordered Banach space $X$. Let $\Omega $ be
an open bounded subset with $0 \in \Omega\cap K$ and
$\overline{\Omega \cap K}\neq K$.  Let $\Omega ^{1}$ be open in
$X$ with $\overline{\Omega ^{1}}\subset \Omega \cap K$. Let
$F:\overline{\Omega \cap K}\rightarrow K$ be a compact map.
 Suppose that
\begin{itemize}
\item[(1)]$Fx\neq \mu x$ for all $x\in\partial( \Omega \cap K)$
and for all $\mu \geq 1$.
\item[(2)] There exists $h\in K\setminus \{0\}$ such that $x\neq Fx+\lambda h$ for all $x\in \partial (\Omega^1 \cap K)$ and all $\lambda
\geq0$.
\end{itemize}
Then $F$ has at least one fixed point $x \in (\Omega \cap
K)\setminus\overline{(\Omega^{1}\cap K)}$.\\
Denoting by $i_K(F,U)$ the fixed point index of $F$ in some
$U\subset X$, $$i_{K}(F,\Omega \cap K)=1 \mbox{ and }
i_{K}(F,\Omega^{1} \cap K)=0\,.$$
 The same result holds if
$$i_{K}(F,\Omega \cap K)=0 \mbox{ and }
i_{K}(F,\Omega^{1} \cap K)=1\,.$$
\end{thm}
Now we fix
 \begin{equation*}\label{ci}
 [a_1,b_1] \subset(0,1),\, \, c_1=\min\{a_1,1-b_1\}, \,\, [a_2,b_2] \subset(0,1] \text{ and } c_2=a_2, 
 \end{equation*} 
and we consider  the cone
 $$\mathcal K:= \mathcal K_{\omega_1}\times \mathcal K_{\omega_2}$$   in $C_{\omega_1}^{1}[0,1]\times C_{\omega_2}^{1}[0,1]$, where $\mathcal K_{\omega_i}$ is
\begin{eqnarray*}\label{cone} \mathcal K_{\omega_i}:=\left\{ w\in C_{\omega_i}^{1}[0,\,1]:w\geq 0,\,\underset{t\in [ a_i,b_i]}%
{\min }w(t)\geq c_i || w|| _{\infty},\,\,  || w|| _{\infty}\geq \|w'\|_{\omega_i}\right\}.
\end{eqnarray*}
We have showed in Section 2 that the cone $\mathcal K$ is total and in ~\cite{cia-pie} that  is $\mathcal T-$invariant.
\\
A \emph{positive solution} of the system (\ref{PDE}) means
a solution $(u,v)\in \mathcal K$ of (\ref{1syst})  such that $\|(u,v)\|> 0$.
\smallskip

In order to use the fixed point index,  we utilize the open bounded sets (relative to $\mathcal K$), for
$\rho_1,\rho_2>0$,
\begin{equation*}
K_{\rho _{1},\rho _{2}}:=\{(w_1,w_2)\in \mathcal K:|| w_1|| <\rho _{1}\ \text{ and }\ || w_2|| <\rho _{2}\},
\end{equation*}
for which holds the  property:
\begin{center}
 $(w_1,w_2)\in\partial K_{\rho _{1},\rho _{2}}$ if and only if
$(w_{1},w_{2})\in \mathcal K$ and for some $i\in \{1,2\}$ $\|w_i\|_{\infty}=\rho_i$ \text{ and } $c_i \rho_i \le
w_i(t)\le \rho_i$ for $t\in[a_i,b_i]$.
\end{center}
\section{Characteristic values of linear operators}
Let $L:X\to X$ be a linear operator on a Banach space $X$.  A number $\lambda$ is said an \textit{eigenvalue} of  $L$ with corresponding eigenfunction $\varphi$ if $\varphi \neq 0$ and ${\lambda}\varphi=L \varphi$. The reciprocals of nonzero
eigenvalues are called \emph{characteristic values} of $L$.
The \textit{spectral radius} of $L$  is given by $r(L):=\displaystyle\lim_{n\to\infty}\|L^n\|^{\frac{1}{n}}$ and its \textit{principal characteristic value} by $\mu(L)=1/r(L)$ .\\
We give the statements of the main tools in this section.
 \begin{thm}{\emph (Krein-Rutman Theorem)} \label{kre}~\cite{kre}\\
Assume that $K$ is a total cone in a real Banach space $X$ and $L:X\to X$  is a compact linear operator such that $L(K)\subset K$ and $r(L)>0$. Then there exists a nonzero element $u \in K$ such that $Lu=r(L) u$.
\end{thm}
\begin{defn}\cite{kra, kra1} A positive bounded linear operator $L: X \to X$ is said $u_0 -$\textit{positive relative to the cone} $K$ if there exists $u_0 \in K \setminus\{0\}$ such that for every  $u\in K \setminus\{0\}$ there are constants $d_{2,u}\geq d_{1,u}> 0$ such that
$$d_{1,u} \,u_0\le Lu\le d_{2,u}\,u_0.$$
\end{defn} 
\begin{thm}{\emph (Comparison Theorem)}\label{comparison}~\cite{kee-tra}\\
Let $K$ be a cone in a Banach space $X$ and let $L, S$ be  bounded linear operators, with $L\leq S$. Assume that at least one of the operators is $u_0-$positive on $K$. If there exist
\begin{enumerate}
\item $u_1\in K\setminus\{0\}$ and $\lambda_1>0$ such that $Lu_1 \geq\lambda_1u_1$;
\item $u_2\in K\setminus\{0\}$ and $\lambda_2>0$ such that $Su_2 \leq \lambda_2u_2$,
\end{enumerate}
then $\lambda_1\le \lambda_2$ and, if $\lambda_1=\lambda_2$, $u_1$ is a scalar multiple of $u_2$.
\end{thm}

In order to state the eigenvalue criteria, consider the linear Hammerstein operator $\mathcal L$ on $C^1_{\omega_1}[0,1]\times C^1_{\omega_2}[0,1]$, associate to the operator $\mathcal T$,  defined by, for $t\in [0,1]$,
\begin{equation*}\label{opL}
\mathcal L(u_1,u_2)(t):=
\left(
\begin{array}{c}
 \displaystyle\int_{0}^{1}k_1(t,s)u_1(s)\,ds\\
\displaystyle\int_{0}^{1}k_2(t,s)u_2(s)\,ds%
\end{array}
\right)
:=
\left(
\begin{array}{c}
\mathcal L_1u_1(t) \\
\mathcal L_2u_2(t)%
\end{array}
\right) .
\end{equation*}
\begin{thm}\label{lcomp}
The operator $\mathcal L$ is compact and map $\mathcal P=\mathcal P_{\omega_1}\times \mathcal P_{\omega_2}$ into $\mathcal K$.
\end{thm}
\begin{proof}
 Note that the operator $\mathcal L$ maps $\mathcal P$ into $\mathcal P$ because the kernels are positive functions; now we show that $\mathcal L$ maps $\mathcal P$ into $\mathcal K$.

By \eqref{k1} and \eqref{k2},
for every $u_i \in \mathcal P_{\omega_i}$ it follows
$$
\mathcal L_iu_i(t) \leq
\int_{0}^{1}\phi_i(s)u_i(s)ds
$$
and therefore $$|| \mathcal L_iu_i|| _{\infty}\leq
\int_{0}^{1}\phi_i(s)u_i(s)ds
<+\infty.$$
On the other hand, we have
\begin{equation*}\label{min}
\min_{t\in [a_i,b_i]}L_iu_i(t) \geq
c_i\int_{0}^{1}\phi_i(s)u_i(s)ds \geq c_i
||\mathcal L_iu_i|| _{\infty}.
\end{equation*}
Now  we prove that, if $u_1\in \mathcal P_{\omega_1}$, it holds
\begin{equation}\label{u'}
\|(\mathcal L_1u_1)'\|_{\omega_1}\leq \|\mathcal L_{1}u_1 \|_{\infty}.
\end{equation}
In fact we have
\begin{align*}
&t(1-t)|(\mathcal L_1u_1)'(t)|=\Big|-t(1-t)\int_0^t
su_1(s)ds
+t(1-t)\int_t^1
(1-s)u_1(s)ds\Big|\\
&\leq t(1-t)\int_0^ts u_1(s)ds+t(1-t)\int_t^1 (1-s)u_1(s)ds\\
&\leq (1-t)\int_0^tsu_1(s)ds+t\int_t^1 (1-s)u_1(s)ds=\mathcal L_1u_1(t)\leq \|\mathcal L_1u_1\|_{\infty}
\end{align*}
and consequently \eqref{u'} holds.

Analogously, for $u_2 \in \mathcal P_{\omega_2}$, we obtain
\begin{align*}
t|(\mathcal L_2u_2)'(t)| &= t\int_{t}^{1}u_2(s)ds\le 
\int_0^t su_2(s)ds+t\int_{t}^{1}u_2(s)ds=\mathcal L_2u_2(t)\leq \|\mathcal L_2u_2\|_{\infty}
\end{align*}
 and therefore we have
 \begin{equation*}\label{v'}
  \|(\mathcal L_2u_2))'\|_{\omega_2}\leq  \|\mathcal L_{2}u_2 \|_{\infty}.
\end{equation*}
Finally, by the properties of the Green's functions $k_i$ and using the
Arz\`{e}la-Ascoli Theorem, it follows that the operator $\mathcal L$ is
compact.
\end{proof}
\smallskip

\noindent
Since $\|w\|=\|w\|_{\infty}$ for each $w \in \mathcal K_{\omega_i}$,  the proof of the following theorem is analogous to the ones in ~\cite{jw-gi-jlms,webb-lan}  and  is reported for completeness.

\begin{thm}\label{specrad}
For $i=1,2$, the spectral radius of $\mathcal L_i$ is nonzero and is an eigenvalue of $\mathcal L_i$ with an eigenfunction in $\mathcal K_{\omega_i}$.
\end{thm}
\begin{proof} For $u_i \in \mathcal K_{\omega_i}$ and  $t \in [a_i,b_i]$ one has
\begin{align*}
&\mathcal L_iu_i(t)\geq c_i\int_{a_i}^{b_i}\phi_i(s)u_i(s)ds\geq
c_i\min_{t \in
[a_i,b_i]}u_i(t)\int_{a_i}^{b_i}\phi_i(s)ds\\
&\geq
c_i^2\|u_i\|_{\infty}\int_{a_i}^{b_i}\phi_i(s)ds=c_i^2\|u_i\|\int_{a_i}^{b_i}\phi_i(s)ds.
\end{align*}
Then we have
\begin{align*}
&\mathcal L_i^2u_i(t)=\int_{0}^{1}k_i(t,s)\mathcal L_iu_i(s)ds\geq \int_{a_i}^{b_i}k_i(t,s)\mathcal L_iu_i(s)ds\\
&\geq c_i^2\|u_i\|\int_{a_i}^{b_i}\phi_i(s)ds\,\int_{a_i}^{b_i}k_i(t,s)ds\geq c_i^3\|u_i\|
\left(\int_{a_i}^{b_i}\phi_i(s)ds\right)^2
\end{align*}
and analogously we get
\begin{align*} \mathcal L_i^nu_i(t)\geq  c_i^{n+1}\|u_i\|
\left( \int_{a_i}^{b_i}\phi_i(s)ds\right)^{n}.
\end{align*}
Thus we obtain
\begin{align*} \|\mathcal L_i^n\|\|u_i\|\geq \|\mathcal L_i^n u_i\| \geq \mathcal L_i^nu_i(t)\geq c_i^{n+1}\|u_i\|
\left( \int_{a_i}^{b_i}\phi_i(s)ds\right)^{n},
\end{align*}
hence we have
\begin{align*} r(\mathcal L_i)=\lim_{n \to +\infty}\|\mathcal L_i^n\|^{\frac{1}{n}}\geq
c_i\int_{a_i}^{b_i}\phi_i(s)ds>0.
\end{align*}
Then, by Theorem \ref{kre}, $r(\mathcal L_i)$ is an eigenvalue of $\mathcal L_i$ with $\varphi_i$ eigenfunction in $\mathcal P_{\omega_i}$ and, as $\mathcal L_i$ maps $\mathcal P_{\omega_i}$ in $\mathcal K_{\omega_i}$, the
eigenfunction $\varphi_i\in \mathcal K_{\omega_i}$.\\ 
\end{proof}
\begin{rem}
Since the kernels $k_i$ satisfy the following symmetry properties, for all $t,s\in [0,1]$,
$$k_1(s,t)=k_1(t,s)\,\,\text{ and }\,\,k_2(1-s,1-t)=k_2(t,s),$$
Corollary 7.5 in ~\cite{jwgi-lms-II} assures that the linear operators $\mathcal L_i$ are  $u_0 -$positive relative to the cone $\mathcal P_{\omega_i}$ and therefore Theorem \ref{comparison} can be used in the sequel. \\
Moreover, by Theorems \ref{total}, \ref{lcomp} and \ref{specrad}, it follows that the operators $\mathcal L_i$ satisfy the hypotheses of Krein-Rutman Theorem.
\end{rem}
\section{Eigenvalue criteria for the existence of positive solutions}
In this section we give some results that determine  relationships between the upper and lower bounds of the nonlinearities $f_i$ on some stripes of $\mathbb R^n$  and the principal characteristic values  of two linear operators associated to $\mathcal T$.\\
Using  the principal characteristic value of $\mathcal L_i$, in the following two theorems we provide conditions assuring that the index of the operator $\mathcal T$ defined in \eqref{operT} is one in some suitable sets.\\
\\
Let $f_i$ be the nonlinearities of the system \eqref{PDE} .
\begin{thm}
\label{idx1aut1}
Assume that
\begin{enumerate}
\item[]  for $i=1,2$ there exist $\varepsilon_i>0$ and $\rho_i>0$ such that  the following condition holds:
\end{enumerate}
\begin{equation*}\label{eq1mu+}
\sup_{\Omega^{\rho_1,\rho_2}}
f_i(r,w_1,w_2,z_1,z_2) \leq
\frac{(\mu(\mathcal L_i)-\varepsilon_i)}{\displaystyle\sup_{t \in [0,1]}p(t)}w_i,
\end{equation*}
where
$$\Omega^{\rho_1,\rho_2}=[
R_0,R_1]\times[0,\rho_1]\times[0,\rho_2]\times[0, +\infty)^2.$$
Then $i_{\mathcal K}(\mathcal T,K_{\sigma_1,\sigma_2})=1$ for $\sigma_i\le\rho_i, \,i=1,2$.
\end{thm}
\begin{proof}
 Let $\sigma_i\le\rho_i$ for $i=1,2$.  In order to show $ i_{\mathcal K}(\mathcal T,K_{\sigma_1, \sigma_2})=1$, we prove that $\mathcal T(u,v)\ne\lambda (u,v)$ for $(u,v)\in\partial K_{\sigma_1,\sigma_2}$ and $\lambda\ge 1$.
 Otherwise there exist $(u,v)\in\partial K_{\sigma_1, \sigma_2}$ and $\lambda\ge1$ such that $\lambda (u,v)=\mathcal T(u,v)$. Thus we have, for $ t\in [0,1]$, 
\begin{align*}
 u(t)\leq&\lambda u(t)=  \mathcal T_1(u,v)(t) = \int_0^1k_1(t,s)p(s)f_1\left(r(s),u(s),v(s),\frac{|u'(s)|}{|r'(s)|},\frac{|v'(s)|}{|r'(s)|}\right)ds \\ 
 & \le  \frac{ (\mu(\mathcal L_1)-\varepsilon_1)}{\displaystyle\sup_{t\in [0,1]}p(t)}\int_0^1k_1(t,s)p(s)u(s)ds\le   (\mu(\mathcal L_1)-\varepsilon_1)\int_0^1k_1(t,s)u(s)ds\\
 &=(\mu(\mathcal L_1)-\varepsilon_1)\mathcal L_1 u(t).
 \end{align*}
Therefore, by the monotone properties of the operator $\mathcal L_1$ we have, for $ t\in [0,1]$, 
\begin{align*}
u(t)\le &
(\mu(\mathcal L_1)-\varepsilon_1)\mathcal L_1[(\mu(L_1)-\varepsilon_1)\mathcal L_1u(t)]=(\mu(\mathcal L_1)-\varepsilon_1)^2\mathcal L_1^2u(t)\le\cdots\\
&\le(\mu(\mathcal L_1)-\varepsilon_1)^n\mathcal L_1^nu(t)\le(\mu(\mathcal L_1)-\varepsilon_1)^n\|\mathcal L_1^n\|\|u\|=(\mu(\mathcal L_1)-\varepsilon_1)^n\|\mathcal L_1^n\|\|u\|_{\infty};
 \end{align*}
thus, taking the norm, we obtain
$$1\le(\mu(\mathcal L_1)-\varepsilon_1)^n\|\mathcal L_1^n\|,$$ and then we get
$$1\le(\mu(\mathcal L_1)-\varepsilon_1)\lim_{n\to\infty}\|\mathcal L_1^n\|^\frac{1}{n}=\frac{\mu(\mathcal L_1)-\varepsilon_1}{\mu(\mathcal L_1)}<1,$$
a contradiction.
\end{proof}
\smallskip

\noindent
\begin{thm}
\label{idx1aut2}
Assume that
\begin{enumerate}
\item[]  for $i=1,2$ there exist $\varepsilon_i>0$ and $\theta_i>0$ such that the following condition holds:
\end{enumerate}
\begin{equation*}\label{eq1muin}
\sup_{E^{\theta_1,\theta_2}}f_{i}(r,w_1,w_2,z_1,z_2) \le
\frac{(\mu(\mathcal L_i)-\varepsilon_i)}{\displaystyle \sup_{t \in [0,1]}p(t)}w_i;
\end{equation*}
$$E^{\theta_1,\theta_2}=[
R_0,R_1]\times[\theta_1,+\infty)\times[\theta_2+\infty)\times[\theta_1, +\infty)\times [\theta_2,+\infty).$$
 Then there exist $\tau_1, \tau_2>0$  such that $i_{\mathcal K}(\mathcal T,K_{\sigma_1, \sigma_2})=1$  for each  $\sigma_i>\tau_i,$ $ i=1,2$.
\end{thm}
\begin{proof}

Since the functions $f_i$ are continuous,  there exist  some constants $N_i$ depending on $\theta_i$  such that
\begin{equation*}
f_i(r,w_1,w_2,z_1,z_2)\le \frac{N_{i}}{\displaystyle\sup_{t \in [0,1]}p(t)} \;\text{ for } \; (r,w_1,w_2, z_1,z_2)\in [R_0,R_1] \times[0,\theta_1]\times[0,\theta_2]\times[0,\theta_1]\times[0,\theta_2].
\end{equation*}
 Hence
\begin{equation}\label{supest}
f_i(r,w_1,w_2,z_1,z_2)\le\frac{(\mu(\mathcal L_i)-\varepsilon_i)}{\displaystyle \sup_{t \in [0,1]}p(t)}w_i+\frac{N_{i}}{\displaystyle\sup_{t \in [0,1]}p(t)} 
\end{equation}
for $r\in [R_0,R_1]$ and $w_1, w_2, z_1,z_2\geq 0$ .\\
 Let $\Id$ be the identity operator. Since for $i=1,2$ the
operators $(\mu(\mathcal L_i)-\varepsilon_i) \mathcal L_i$ have spectral radius less
than one, the operators
$(\Id-(\mu(\mathcal L_i)-\varepsilon_i)\mathcal L_i)^{-1}$ exist and are bounded.
Moreover, from the Neumann series expression,
$$
(\Id-(\mu(\mathcal L_i)-\varepsilon_i) \mathcal L_i)^{-1}=\sum_{k=0}^\infty((\mu(\mathcal L_i)-\varepsilon_i) \mathcal L_i)^k
$$
it follows that  $(\Id-(\mu(\mathcal L_i)-\varepsilon_i) \mathcal L_i)^{-1}$ map $\mathcal P$ into $\mathcal P$, since the operators $\mathcal L_i$ have this property.

Take for $i=1,2$
$$
C_i:=N_i\int_{0}^{1}\phi_i(s)ds \,\,\text{ and }\,\,\tau_i:=\|(\Id-(\mu(\mathcal L_i)-\varepsilon_i) L_i)^{-1} C_i \|_{\infty} .
 $$
 Now we prove that, for $\sigma_i>\tau_i$, $\mathcal T(u,v)\ne\lambda (u,v)$ for all $(u,v)\in\partial K_{\sigma_1,\sigma_2}$ and $\lambda\ge 1$, which implies $ i_{\mathcal K}(\mathcal T,K_{\sigma_1,\sigma_2})=1$.
 Otherwise there exist $(u,v)\in\partial K_{\sigma_1,\sigma_2}$ and $\lambda\ge 1$ such that $\lambda (u,v)=\mathcal T(u,v)$. Suppose that $\|u\|=\|u\|_{\infty}=\sigma_1$ and $\|v\|=\|v\|_{\infty}\leq \sigma_2$. \\
 From the  inequality \eqref{supest}, it follows that, for $t\in [0,1]$,
\begin{align*}
u(t)\leq\lambda &u(t)=  \mathcal T_1(u,v)(t)  \le \int_0^1k_1(t,s)p(s)f_1\left(r(s),u(s),v(s),\frac{|u'(s)|}{|r'(s)|},\frac{|v'(s)|}{|r'(s)|}\right)ds\\
\le & \frac{(\mu(\mathcal L_1)-\varepsilon_1)}{\displaystyle \sup_{t \in [0,1]}p(t)}\int_0^1k_1(t,s)p(s)u(s)ds+\frac{N_1}{\displaystyle \sup_{t \in [0,1]}p(t)}\int_0^1k_1(t,s)p(s)ds \\
\le & (\mu(\mathcal L_1)-\varepsilon_1)\int_0^1k_1(t,s)u(s)ds+N_1\int_0^1\phi_1(s)ds=
(\mu(\mathcal L_1)-\varepsilon_1)\mathcal L_1 u(t)+C_1,
\end{align*}
which implies
$$
 (\Id-(\mu(\mathcal L_1)-\varepsilon_1) \mathcal L_1)u(t)\le C_1.
 $$
 Since $(\Id-(\mu(\mathcal L_1)-\varepsilon_1) \mathcal L)^{-1}$ is non-negative, it follows that, for $t\in [0,1]$,
$$
u(t)\le (\Id-(\mu(\mathcal L_1)-\varepsilon_1) \mathcal L_1)^{-1} C_1.
$$
Consequently $\sigma_1=\|u\|_{\infty}\le \tau_1$, a contradiction.\\
The case $\|u\|_{\infty}\le\sigma_1$ and $\|v\|_{\infty}= \sigma_2$ is obtained in a similar way.
\end{proof}
\begin{rem}  In  ~\cite{jw-lms,webb-lan}, under suitable assumptions on $k$, Webb and Lan prove that the operator $L$ defined by
$$Lu(t)=\int_0^1k(t,s)u(s)ds,$$
satisfies  the following inequality $m \leq \mu(L) \le M(a,b)$, where $m$ and $M(a,b)$ are defined similary to \eqref{m}  and \eqref{M1}.\\
Moreover in ~\cite{webb-lan}, in the setting of the space of continuous functions,  the authors obtained that $$\mu(\mathcal L_1)=\pi^2 \,\,\text{ and }\,\,\mu(\mathcal L_2)=\frac{\pi^2}{4};$$
 the same results hold in the space $C^1[0,1]$. 
\end{rem}
\smallskip

\noindent
In the zero index calculation of the operator $\mathcal T$  it is more convenient to use  the linear operators  $\overline{\mathcal L}_i:C^1_{\omega_i}[a_i,b_i]\to C^1_{\omega_i}[a_i,b_i]$ defined by, for $t \in [a_i,b_i]$,
$$\overline{\mathcal L}_iu_i(t)=\int_{a_i}^{b_i}k_i(t,s)u_i(s)ds,$$ 
that have the same properties of  the operators $\mathcal L_i$.
\begin{thm} \label{idx0aut1} Assume that
\begin{enumerate}
\item[] there exist $\varepsilon_i>0$ and $\rho_i>0$ such that the following condition holds for some $i=1,2$:
\end{enumerate}
\begin{equation}\label{eqmu+}
\inf_{B_i^{\rho_1,\rho_2}}
f_i(r,w_1,w_2,z_1,z_2) \geq
\frac{(\mu(\overline{\mathcal L}_i)+\varepsilon_i)}{\displaystyle\inf_{t \in [a_i,b_i]}p(t)}w_i, 
\end{equation}
where
$$B_i^{\rho_1,\rho_2}=[
\min\{r(a_i),r(b_i)\},\max\{r(a_i),r(b_i)\}]\times[0,\rho_1]\times[0,\rho_2]\times[0, +\infty)^2.$$
Then $i_{\mathcal K}(\mathcal T,K_{\sigma_1,\sigma_2})=0$ for $\sigma_i\leq \rho_i,\, i=1,2.$
\end{thm}
\begin{proof}
 Let $\sigma_i\leq 
\rho_{i}$ and  let $\varphi_i\in \mathcal K_{\omega_i}$ be the eigenfunction of $\overline{\mathcal L}_i$ with $\|\varphi_i\|=\|\varphi_i\|_{\infty}=1$ corresponding to the eigenvalue $1/\mu(\overline{\mathcal L}_i)$. Now we show that $(u,v)\ne
\mathcal  T(u,v)+\lambda(\varphi_1,\varphi_2)$ for all $(u,v)$ in $\partial
K_{\sigma_1,\sigma_2}$ and $\lambda\geq 0$ which implies that
$ i_{\mathcal K}(\mathcal T,K_{\sigma_1,\sigma_2})=0$.\\
Assume, on the contrary, that there exist $(u,v)\in\partial
K_{\sigma_1,\sigma_2}$ and $\lambda\geq0$ such that
$(u,v)=\mathcal T(u,v)+\lambda(\varphi_1,\varphi_2)$.\\ Two
cases are distincts. Firstly we discuss the case $\lambda>0$. Suppose that
\eqref{eqmu+} holds for $i=1$. This implies that, for $t\in [a_1,b_1]$,
\begin{align*}
  u(t)\geq & \int_{a_{1}}^{b_{1}} k_{1}(t,s)p(s)f_1\left(r(s),u(s),v(s),\frac{|u'(s)|}{|r'(s)|},\frac{|v'(s)|}{|r'(s)|}\right)ds+\lambda \varphi_1(t) \\
  >&\frac{\mu( \overline{\mathcal L}_1) }{\displaystyle \inf_{t \in [a_1,b_1]}p(t)}\int_{a_{1}}^{b_{1}} k_1(t,s)p(s)u(s)ds+\lambda \varphi_1(t) \\
  &\geq \mu( \overline{\mathcal L}_1) \int_{a_{1}}^{b_{1}} k_1(t,s)u(s)ds+\lambda \varphi_1(t) = \mu(\overline{\mathcal L}_1)  \overline{\mathcal L}_1u(t)+\lambda\varphi_1(t).
 \end{align*}
Moreover  $u(t)>\lambda\varphi_1(t)$ for $t\in [a_1,b_1]$; then $
\overline{\mathcal L}_1u(t)\ge\lambda \overline{\mathcal L}_1\varphi_1(t)= \dfrac{\lambda}{\mu(
\overline{\mathcal L}_1)}\varphi_1(t)$ and we obtain
$$
u(t)>\mu( \overline{\mathcal L}_1)
\overline{\mathcal L}_1u(t)+\lambda\varphi_1(t)\ge2\lambda\varphi_1(t).
$$
By iteration, it follows that, for $ t\in[a_1,b_1]$, 
$$
u(t)> n\lambda\varphi_1(t)  \text{   for every  } n\in\mathbb {N},
$$
 a contradiction because $\|u\|\leq\sigma_1$.\\
 Now we consider the case $\lambda=0$.  We have, for $t\in [a_1,b_1]$,
\begin{align*}
 u(t)
  \geq& \int_{a_{1}}^{b_{1}} k_{1}(t,s)p(s)f_1\left(r(s),u(s),v(s),\frac{|u'(s)|}{|r'(s)|},\frac{|v'(s)|}{|r'(s)|}\right)ds \geq(\mu(  \overline{\mathcal L}_1)+\varepsilon_1)
  \overline{\mathcal L}_1u(t)
\end{align*}
and, consequently, we obtain
$$\overline{\mathcal L}_1u(t)\leq \frac{1}{\mu(  \overline{\mathcal L}_1)+\varepsilon_1}u(t).
$$
Then, by Comparison Theorem~\ref{comparison}, it follows that $r({\overline{\mathcal L}}_1)\leq
\dfrac{1}{\mu( \overline{\mathcal L}_1)+\varepsilon_1}$ and thus we get
$$\mu(
\overline{\mathcal L}_1)+\varepsilon_1\leq\frac{1}{r(\overline{\mathcal L}_1)}= \mu(\overline{\mathcal L}_1),$$ a contradiction. \par
 \end{proof}
 \begin{rem}
 As in  ~\cite{webb-lan}, we obtain by direct calculations that
 $$
 \mu(\overline{\mathcal L_1})=\frac{\pi^2}{(b_1-a_1)^2} \,\,\text{ and }\,\,\mu(\overline{\mathcal L_2})=\frac{\pi^2}{4(b_2-a_2)^2} .
 $$
 \end{rem}
 \bigskip

 \noindent
The difference between Theorem \ref{idx0aut1} and the following Theorem  consists in the fact that in Theorem \ref{idx0aut3}  the lower bound of the $f_i$ is calculate for $w_i$ enough  far from the zero.\\
\bigskip

\noindent
  \begin{thm} \label{idx0aut3}
Assume that
\begin{enumerate}
\item[]  for $i=1,2$ there exist $\varepsilon_i>0$ and $\theta_i>0$ such that  the following condition holds:
\end{enumerate}
\begin{equation}\label{eqmu+in}
\inf_{D_i ^{\theta_1,\theta_2}}
f_i(r,w_1,w_2,z_1,z_2) \geq
\frac{(\mu(\overline{\mathcal L}_i)+\varepsilon_i)}{\displaystyle\inf_{t \in [a_i,b_i]}p(t)}w_i,\end{equation}
where
$$D_1^{\theta_1,\theta_2}=[
\min\{r(a_1),r(b_1)\},\max\{r(a_1),r(b_1)\}]\times[c_1\theta_1,+\infty)\times[0,+\infty)^3,$$
$$
D_2^{\theta_1,\theta_2}=[
\min\{r(a_2),r(b_2)\},\max\{r(a_2),r(b_2)\}]\times[0,+\infty)\times[c_2\theta_2,+\infty)\times[0,+\infty)^2.$$
Then  $i_{\mathcal K}(\mathcal T,K_{s_1,s_2})=0$ for $s_i\geq \theta_i$, $i=1,2$.
\end{thm}
\begin{proof}
 Let $s_i\geq \theta_i$. We prove that
$(u,v)\ne \mathcal T(u,v)+\lambda(\varphi_1,\varphi_2)$ for all $(u,v)$ in
$\partial K_{s_1,s_2}$ and $\lambda\geq 0$,
where $\varphi_i\in \mathcal  K_{\omega_i}$ is the eigenfunction associated to $r(\overline{\mathcal L}_i)$ as in Theorem \ref{idx0aut1}, which implies that
$ i_{\mathcal K}(\mathcal T ,K_{s_1,s_2})=0$.\\
Assume, on the contrary, that there exist $(u,v)\in\partial
K_{s_1,s_2}$ and $\lambda\geq0$ such that
$(u,v)=\mathcal T(u,v)+\lambda(\varphi_1,\varphi_2)$. Suppose that
$\|u\|=s_1$ and $\|v\|\leq s_2$.\\
Then, for $t\in[a_1,b_1]$, $u(t)\ge c_1\|u\|=c_1 s_1\geq c_1 \theta_1$,  thus condition \eqref{eqmu+in} holds. Hence we obtain, for $t\in[a_1,b_1]$,
$$p(t)f_1\left(r(t),u(t),v(t),\frac{|u'(t)|}{|r'(t)|},\frac{|v'(t)|}{|r'(t)|}\right)\geq(\mu( \overline{\mathcal L}_1)+\varepsilon_1)u(t).$$   Proceeding as in the proof of Theorem \ref{idx0aut1} in the the case $\lambda>0$, this implies that, for $t\in [a_1,b_1]$,
$$
u(t)>\mu(\overline{\mathcal L}_1)
\overline{\mathcal L}_1u(t)+\lambda\varphi_1(t)\ge2\lambda\varphi_1(t).
$$
Then $u(t)> n\lambda\varphi_1(t)$ for every $n\in\mathbb {N}$, a contradiction because $\|u\|=s_1$.\\
The proof in the case $\|u\|\leq s_1$ and $\|v\|=s_2$  is analogous and the case $\lambda=0$ is treated as in Theorem \ref{idx0aut1}.
 \end{proof}
\noindent
Using Theorem \ref{idx1aut1} and Theorem \ref{idx0aut3},  the following existence result of positive radial solution for the system \eqref{PDE} holds.
\\
\begin{thm} \label{existence}Assume
that
\begin{enumerate}
\item[]  for $i=1,2$, there exist $\varepsilon_i, \eta_i>0$, $\rho_i, \theta_i>0$, with $\rho_{i}\leq c_i \theta_i$, such that  the following conditions hold:
\end{enumerate}
\begin{equation*}\label{eqmu+2}
\sup_{\Omega^{\rho_1,\rho_2}}
f_i(r,w_1,w_2,z_1,z_2) \leq
\frac{(\mu(\mathcal L_i)-\varepsilon_i)}{\displaystyle\sup_{t \in [0,1]}p(t)}w_i,
\end{equation*}
and
\begin{equation*}\label{eq1mu+}
\inf_{D_i^{\theta_1,\theta_2}}
f_i(r,w_1,w_2,z_1,z_2) \geq
\frac{(\mu(\overline{\mathcal L}_i)+\eta_i)}{\displaystyle\inf_{t \in [a_i,b_i]}p(t)}w_i,
\end{equation*}
where
$\Omega^{\rho_1,\rho_2}$ and $D_i^{\theta_1,\theta_2}$ are as in Theorem \ref{idx1aut1} and Theorem \ref{idx0aut3}.\\
Then the system (\ref{PDE}) has at least one positive radial
solution.
\end{thm}
\smallskip

\noindent
The index results in this Section  can be carefully combined  in order to establish results on existence of multiple positive solutions for the system \eqref{PDE}. We refer to ~\cite{lan-lin-na} for similar statements.\\

\begin{ex}
Theorem ~\ref{existence} can be applied when the nonlinearities $f_i$ are of the type
\begin{equation*}
f_i(|x|,u,v,|\nabla u|, |\nabla v|)=(\delta_i u^{\alpha_i}+\gamma_i v^{\beta_i})h_i(|x|,u,v,|\nabla u|, |\nabla v|)
\end{equation*}
with $h_i$ continuous functions bounded  by a strictly positive constant, $\alpha_i,\beta_i>1$ and  $\delta_i,\gamma_i$ suitable positive constants.\\
For example, one can consider the following system
\begin{gather}\label{ellbvpex}
\begin{cases}
&-\Delta u =\left(1+\frac{2}{\pi}\arctan(|x|^2+|\nabla v|^2)\right)\,u^2 \text{ in } \Omega, \\
&-\Delta v=\frac{4}{\pi}\arctan\left(1+|\nabla u|^2+|\nabla v|^2\right)\,v^2\text{ in } \Omega,\\
&u=0 \text{ on }\partial \Omega,\\
&v=0 \text{ on }|x|=1 \text{ and }\displaystyle\frac{\partial
v}{\partial r}=0 \text{ on }|x|=e,
\end{cases}
\end{gather}
where $\Omega=\{ x\in\mathbb{R}^3 : 1<|x|<e\}$.\\
By direct computation, we obtain $\displaystyle\sup_{t \in [0,1]}p(t)=e^2(e-1)^2$ and, fixed $[a_1,b_1]=\left[\frac{1}{4},\frac{3}{4}\right],\,[a_2,b_2]=\left[\frac{1}{2},1\right]$, we have $c_1=\frac{1}{4},\,c_2=\frac{1}{2}$,
\begin{align*}
&\inf_{t \in \left[\frac{1}{4},\frac{3}{4}\right]}p(t)=\frac{e^2(e-1)^2}{\left(e-\frac{e-1}{4}\right)^4},\,\,\,\,\,\,\inf_{t \in \left[\frac{1}{2},1\right]}p(t)=\frac{e^2(e-1)^2}{\left(e-\frac{e-1}{2}\right)^4}.
\end{align*}
With the choice of $\rho_1=1/10,\,\,\rho_2=1/25\,, \theta_1=200,\,\theta_2=50$,
we obtain
\begin{align*}
&\sup_{\Omega^{\rho_1,\rho_2}}\, f_1= 2\rho_1^2=  0.02<0.045=\frac{\pi^2}{\displaystyle\sup_{t \in[0,1]}p(t)}\,\rho_1;\\
&\sup _{\Omega^{\rho_1,\rho_2}} f_2= 2\rho_2^2=0.0032<0.0045=\frac{\pi^2}{\displaystyle4\sup_{t \in
[0,1]}p(t)}\,\rho_2,\\
&\inf_{D_1^{\theta_1,\theta_2}}  f_1= \frac{1}{16}\theta_1^2=2500>2482.65=\frac{4\pi^2}{\displaystyle\inf_{t \in
\left[\frac{1}{4},\frac{3}{4}\right]}p(t)}\,\frac{\theta_1}{4};\\
&\inf_{D_2^{\theta_1,\theta_2}}\, f_2= \frac{1}{4}\theta_2^2=625>540.47=\frac{\pi^2}{\displaystyle\inf_{t \in
\left[\frac{1}{2},1\right]}p(t)}\,\frac{\theta_2}{2};
\end{align*}
consequently the nonlinearities $f_i$ satisfy Theorem \ref{existence} and the system \eqref{ellbvpex} admits at least one positive radial solution.\\
We conclude by noting that, with this choice of radius, $f_1$ does not satisfy the hypotheses of Theorem \ref{ellyptic} because
$$\sup_{\Omega^{\rho_1,\rho_2}}\, f_1=2\rho_1^2=  0.02>0.014=\frac{m_1}{\displaystyle\sup_{t \in[0,1]}p(t)}\,\rho_1.$$
\end{ex}

\bigskip


\begin{thebibliography}{00}
\bibitem{aga-ore-yan} 
R. P. Agarwal, D. O'Regan and B. Yan, Multiple positive solutions of singular Dirichlet second order boundary-value problems with derivative dependence, \textit{J. Dyn. Control Syst.}, \textbf{15} (2009), 1--26.



\bibitem{Amann-rev} H. Amann,
Fixed point equations and nonlinear eigenvalue
problems in ordered Banach spaces, \textit{SIAM. Rev.}, \textbf{18} (1976),
620--709.

\bibitem{ave-graef-liu} R. I. Avery, J. R. Graef and X. Liu,
 Compression fixed point theorems of operator type,
 \textit{J. Fixed Point Theory Appl.}, \textbf{17}  (2015), 83--97.


\bibitem{ave-mo-tor} D. Averna, D. Motreanu and E. Tornatore,
Existence and asymptotic properties for quasilinear elliptic equations with gradient dependence,
 \textit{Appl. Math. Lett.},  \textbf{61} (2016), 102--107.

\bibitem{bue-er-zu-fe} H. Bueno, G. Ercole, A. Zumpano and W. M. Ferreira,
 Positive solutions for the $p-$Laplacian with dependence on the gradient,
\textit{Nonlinearity}, \textbf{25} (2012), 1211--1234.


\bibitem{cia-pie}
F. Cianciaruso and P. Pietramala,  Semilinear elliptic systems with dependence on the gradient, \textit{Mediterr. J. Math.}, \textbf{15}  (2018),  Art. 152, 13 pp.

\bibitem{defig-sa-ubi} D. G.  De Figueiredo, J. S{\'a}nchez and P. Ubilla,
Quasilinear equations with dependence on the gradient,
 \textit{Nonlinear Anal.}, \textbf{71} (2009), 4862--4868.
 
\bibitem{defig-ubi} D. G. De Figueiredo and P. Ubilla,
Superlinear systems of second-order ODE's,
  \textit{Nonlinear Anal.}, \textbf{68}, no. 6 (2008), 1765--1773.
  
  \bibitem{dolo3} J. M. do {\'O}, J. S{\'a}nchez, S. Lorca and P. Ubilla,
Positive solutions for a class of multiparameter ordinary elliptic systems,
\textit{J. Math. Anal. Appl.}, \textbf{332} (2007),  1249--1266.

  \bibitem{erbe}L. Erbe,
Eigenvalue criteria for existence of positive solutions to nonlinear boundary value problems,
  \textit{Math. Comput. Modelling}, \textbf{32}, no. 5-6 (2000),  529--539.
  
\bibitem{fa-mi-pe} L. F. O. Faria, O. H.  Miyagaki and F. R. Pereira,
 Quasilinear elliptic system in exterior domains with dependence on the gradient,
   \textit{Math. Nachr.}, \textbf{287} (2014), 61--373.
   
  



\bibitem{nirenberg} B. Gidas, W.M. Ni and L. Nirenberg,
 Symmetry and related properties via the maximum principle,
\textit{Comm. Math. Phys.}, {\bf 68} (1979), 209--243.

\bibitem{gra-kon-min} J. Graef, L. Kong and F. Minh\'{o}s, Generalized Hammerstein equations and applications, \textit{Results Math.}, {\bf 72} (2017), 369--383.

 \bibitem{guo-ge}  Y. Guo and W. Ge,
 Positive solutions for three-point boundary value problems with dependence on the first order derivative,
 \textit{J. Math. Anal. Appl.}, \textbf{290} (2004), 291--301.
 
\bibitem{guolak} D. Guo and V. Lakshmikantham,
\textit{Nonlinear Problems in Abstract Cones}, Academic Press, Boston, 1988.

\bibitem{inf-min} G. Infante and F. Minh\'{o}s, Nontrivial solutions of systems of Hammerstein integral equations with first derivative dependence, \textit{Mediterr. J. Math.}, {\bf 14} (2017), Art. 242, 18 pp.



   
\bibitem{inf-pie} G. Infante and P. Pietramala,
 Nonzero radial solutions for a class of elliptic systems with nonlocal BCs on annular domains,
 \textit{NoDEA Nonlinear
Differential Equations Appl.}, \textbf{22} (2015), 979--1003.

\bibitem{janko} T. Jankowski,
Nonnegative solutions to nonlocal boundary value problems for systems of second-order differential equations dependent on the first-order derivatives,
\textit{Nonlinear Anal.}, \textbf{87} (2013), 83--"101.

\bibitem{kee-tra}
M.S. Keener and C.C. Travis, Positive cones and focal points for a class of nth order differential equations, \textit{Trans. Amer. Math. Soc.},  \textbf{237} (1978), 331--351.

\bibitem{kra}
M.A. Krasnosel'ski\u\i{}, Positive Solutions of Operator Equations, P. Noordhoff Ltd. Groningen, 1964.

\bibitem{kra1}
Krasnosel'ski\u\i{}, Topological Methods in the Theory of Nonlinear Integral Equations, The Macmillan Co., New York, 1964.


\bibitem{kre} M. G. Krein and M. A. Rutman,
Linear operators leaving invariant a cone in a Banach space,
\textit{Uspekhi Mat. Nauk.}, \textbf{23} (1948), 3--95(in Russian);  \textit{Amer. Math. Soc. Transl.} \textbf{26} (in English).


\bibitem{lan}
K. Q. Lan,  Eigenvalues of semi-positone Hammerstein integral equations and applications to boundary value problems,  \textit{Nonlinear Anal.}  \textbf{71} no. 12 (2009),  5979--5993.

\bibitem{lan1}
K. Q. Lan, 
Nonzero positive solutions of systems of elliptic boundary value problems, \textit{Proc. Amer. Math. Soc.} \textbf{139} (12) (2011), 4343--4349.


\bibitem{lan-lin}
K. Q. Lan and W. Lin,
Multiple positive solutions of systems of Hammerstein integral equations with applications to fractional differential equations,
\textit{J. Lond. Math. Soc. } \textbf{83} (2) (2011), 449--469.

\bibitem{lan-lin-na} K. Q. Lan and W. Lin,
Positive solutions of systems of singular Hammerstein integral equations with applications to semilinear elliptic equations in annuli,
\textit{Nonlinear Anal.}, \textbf{74} (2011), 7184--7197.


\bibitem{li}
Y. Li,  Abstract existence theorems of positive solutions for nonlinear boundary value problems,  \textit{Nonlinear Anal.},  \textbf{57} no. 2(2004), 211--227.

\bibitem{liu-li} 
Z. Liu and F. Li, Multiple positive solutions of nonlinear two-point boundary value problems, \textit{J. Math. Anal. Appl.},  \textbf{203}, no. 3 (1996), 610--625.

\bibitem {min-desou} F. Minh\'{o}s and R. de Sousa, Existence of solution for functional coupled systems with full nonlinear terms and applications to a coupled mass-spring model, \textit{Differ. Equ. Appl.}, {\bf 9} (2017), 433--452.
   	

\bibitem{min-sou} F. Minh\'{o}s and R.  de Sousa, On the solvability of third-order three point systems of differential equations with dependence on the first derivative, \textit{Bull. Braz. Math. Soc. (N.S.)}, {\bf 48} (2017), 485--503.
 		


\bibitem{singh} G. Singh,
 Classification of radial solutions for semilinear elliptic systems with nonlinear gradient terms,
 \textit{Nonlinear Anal.}, \textbf{129} (2015), 77--103.

\bibitem{jw-lms} J. R. L. Webb,
Solutions of nonlinear equations in cones and positive linear operators,
\textit{J. Lond. Math. Soc.}, \textbf{82} (2010), 420--436.


\bibitem{jw-gi-jlms} J.R.L. Webb and G. Infante,
Positive solutions of nonlocal boundary value problems: a unified approach,
\textit{J. London Math. Soc.}, \textbf{74} (2006), 673--693.

\bibitem{jwgi-lms-II}
J. R. L. Webb and G. Infante,
Nonlocal boundary value problems of arbitrary order, \textit{J. London Math. Soc.}, \textbf{79} (2009), 238--258.

\bibitem{webb-lan} J. R. L. Webb and K. Q. Lan,
Eigenvalue criteria for existence of multiple positive solutions of nonlinear boundary value problems of local and nonlocal type,
\textit{Topol. Methods Nonlinear Anal.}, \textbf{27} (2006), 91--115.



\bibitem{yang-kong} Z. Yang and L. Kong,
 Positive solutions of a system of second order boundary value problems involving first order derivatives via $\mathbb R^n_+$-monotone matrices, \textit{Nonlinear Anal.}, \textbf{75} (2012), 2037--2046.

\bibitem{zha-sun}
G. Zhang and J. Sun, Positive solutions of m-point boundary value problems, \textit{J. Math. Anal. Appl.}, \textbf{291} no. 2 (2004),  406--418. 

\bibitem{zima} M. Zima,
 Positive solutions of second-order non-local boundary value problem with singularities in space variables,
\textit{Bound. Value Probl.}, \textbf{2014} (2014): 200, 9 pp.
\end{thebibliography}
\end{document}